\documentclass[11pt]{article}
\usepackage{amssymb}
\usepackage{mathrsfs}
\usepackage{amsmath}
\usepackage[dvips]{graphicx}
\usepackage{latexsym,bm}
\usepackage{amsmath,amssymb,amsfonts}
\usepackage{indentfirst}
\usepackage{amsthm}
\usepackage{cases}
\usepackage{pifont}
\allowdisplaybreaks
\author{Jun-Ming Zhu$^{a,b}$\footnote{E-mail address:
 junming\_zhu@163.com }
\\
$^{a}$Department of Mathematics, Luoyang Normal
University,\\
Luoyang City, Henan Province 471022, China
\\
$^{b}$Department of Mathematics, East China Normal University,\\
Shanghai 200241, China }
\title{An alternate circular  summation formula of  \\theta functions
and its applications \footnote{
This research is supported by the
 Natural Science Foundation of China
 (Grant No. 11171107 ) 
and the Foundation of Fundamental and Advanced Research of Henan Province (Grant No. 112300410024 ).}
 }
\date{}
\newtheorem{lem}{\quad\textbf{\Large Lemma}}[section]
\newtheorem{thm}[lem]{\quad\textbf{\Large Theorem}}
\newtheorem{cor}[lem]{\quad\textbf{\Large Corollary}}
\newtheorem{prop}[lem]{\quad\textbf{\Large Propsotion}}
\begin{document}
 \maketitle
 \setcounter{section}{0}
\begin{abstract}
We prove a general alternate circular summation formula of theta
functions, which implies a great deal of theta-function identities.
In particular, we recover several identities in Ramanujan's Notebook
from this identity. We also obtain two formulaes for
$(q;q)_\infty^{2n}$.


\end{abstract}

\section{Introduction }
Throughout  we put $q=e^{2\pi i\tau}$, where $\mbox{Im}\ \tau>0$. As
usual, the Jacobi theta functions $\theta_k(z|\tau)$ for $k=1,
2,3,4$ are defined as follows:
\begin{eqnarray*}     
\theta_1(z|\tau) &=&-iq^{1\over8}\sum\limits_{n=-\infty}^{\infty}
                           (-1)^nq^{n(n+1)\over2}e^{(2n+1)iz}
                  ,\\
\theta_2(z|\tau)&=&q^{1\over8}\sum\limits_{n=-\infty}^{\infty}
                      q^{n(n+1)\over2}e^{(2n+1)iz}
                 ,\\
  \theta_3(z|\tau)&=&\sum\limits_{n=-\infty}^{\infty}
                      q^{n^2\over2}e^{2niz}
                      , \\
  \theta_4(z|\tau)&=&\sum\limits_{n=-\infty}^{\infty}
                      (-1)^nq^{n^2\over2}e^{2niz}.
\end{eqnarray*}
To carry out our work, we need some notations and basic facts about
the Jacobi theta functions. We use the familiar notation
$$
 (z;q)_\infty=\prod_{n=0}^\infty (1-zq^n)
$$
and sometimes write
$$
(a,b,\cdots,c;q)_\infty = (a;q)_\infty
(b;q)_\infty\cdots(c;q)_\infty.
$$
Using the well-known Jacobi product identity \cite[p. 35, Entry
19]{berndt}
\begin{equation} \label{triple}
f(a,b):=\sum_{n=-\infty}^{\infty}a^{n(n+1)\over2}b^{n(n-1)\over2}
=(ab,-a,-b;ab)_\infty,
\end{equation}
 we can deduce
the infinite product representations for theta functions, namely,
\begin{eqnarray}       \label{thetap}
\begin{array}{ll}
\theta_1(z|\tau)=iq^{1\over8}e^{-iz}(q,e^{2iz},qe^{-2iz};q)_\infty~,
&\theta_2(z|\tau)=q^{1\over8}e^{-iz}(q,-e^{2iz},-qe^{-2iz};q)_\infty~,\\
\theta_3(z|\tau)=(q,-q^{1\over2}e^{2iz},-q^{1\over2}e^{-2iz};q)_\infty~,
&\theta_4(z|\tau)=(q,q^{1\over2}e^{2iz},q^{1\over2}e^{-2iz};q)_\infty~. \\
\end{array}
\end{eqnarray}
Employing the above identities, we can easily get the following
relations:
\begin{eqnarray}  \label{pi}
\begin{array}{lcl}
\theta_1(z+\pi|\tau)=-\theta_1(z|\tau),&&
\theta_2(z+\pi|\tau)=-\theta_2(z|\tau),\\
\theta_3(z+\pi|\tau)=\theta_3(z|\tau),&&
\theta_4(z+\pi|\tau)=\theta_4(z|\tau)
\end{array}
\end{eqnarray}
and
 \begin{eqnarray}       \label{pitau}
\begin{array}{lcl}
\theta_1(z+\pi\tau|\tau)=-q^{-\frac{1}{2}}e^{-2iz}\theta_1(z|\tau),&&
 \theta_2(z+\pi\tau|\tau)=q^{-\frac{1}{2}}e^{-2iz}\theta_2(z|\tau),\\
   \theta_3(z+\pi\tau|\tau)=q^{-\frac{1}{2}}e^{-2iz}\theta_3(z|\tau),&&
   \theta_4(z+\pi\tau|\tau)=-q^{-\frac{1}{2}}e^{-2iz}\theta_4(z|\tau).
\end{array}
\end{eqnarray}
We also have
\begin{eqnarray}       \label{pi/2}
\begin{array}{lcl}
\theta_1(z+{\pi\over2}|\tau)=\theta_2(z|\tau),&&
 \theta_2(z+{\pi\over2}|\tau)=-\theta_1(z|\tau),\\
\theta_3(z+{\pi\over2}|\tau)=\theta_4(z|\tau),&&
 \theta_4(z+{\pi\over2}|\tau)=\theta_3(z|\tau) \\
\end{array}
\end{eqnarray}
and
\begin{eqnarray}       \label{pitau/2}
\begin{array}{lcl}\theta_1(z+{\pi\tau\over2}|\tau)=iq^{-{1\over8}}e^{-iz}
\theta_4(z|\tau),&&
 \theta_2(z+{\pi\tau\over2}|\tau)=q^{-{1\over8}}e^{-iz}\theta_3(z|\tau),\\
\theta_3(z+{\pi\tau\over2}|\tau)=q^{-{1\over8}}e^{-iz}\theta_2(z|\tau),&&
 \theta_4(z+{\pi\tau\over2}|\tau)=iq^{-{1\over8}}e^{-iz}\theta_1(z|\tau).
\end{array}
\end{eqnarray}

We also need the following special case of $f(a,b)$:
\begin{eqnarray} \label{sf}
\varphi(q):=f(q,q),~~\psi(q):=f(q,q^3).
\end{eqnarray}
Definitions (\ref{sf})
 can be found in \cite[p.
36--37, Entry 22]{berndt}. The following properties of $\varphi(q)$ and $\psi(q)$ can be verified through simple computations.
\begin{eqnarray}
\varphi(q)&=&\theta_3(0|2\tau)=(q^2,-q,-q;q^2)_\infty,\\
\varphi(-q)&=&\theta_4(0|2\tau)=(q^2,q,q;q^2)_\infty
=(q;q)_\infty(q;q^2)_\infty =\frac{(q;q)_\infty}{(-q;q)_\infty},\\
\psi(q)&=&\theta_2(\pi\tau|4\tau)=\theta_3(\pi\tau|4\tau)
={1\over2}~q^{-{1\over8}}\theta_2(0|\tau)
=(q,-q,-q;q)_\infty=\frac{(q^2;q^2)_\infty}{(q;q^2)_\infty},~~
\\
\psi(-q)&=&-i\theta_1(\pi\tau|4\tau)
=\theta_4(\pi\tau|4\tau)=(q;q)_\infty(-q^2;q^2)_\infty
=\frac{(q^2;q^2)_\infty}{(-q;q^2)_\infty}. \label{ssf}
\end{eqnarray}

On page 54 of his Lost Notebook \cite{ram}, Ramanujan recorded the
following statement (translated here in terms of
$\theta_3(z|\tau)$).
\begin{thm}\label{ram}
For any positive integer $n\geq 2$,
\begin{equation}
\sum_{k=0}^{n-1}q^{k^2}e^{2k iz}\theta^n_3(z+k\pi\tau|n\tau)=
\theta_3(z|\tau)F_n(\tau).\label{1eq5}
\end{equation}
When $n\geq 3$,
\begin{equation}
F_n(\tau)=1+2nq^{n-1}+\cdots. \label{1eq6}
\end{equation}
\end{thm}

The proof of (\ref{1eq5}) was first given by Rangachari \cite{ran},
and then, by Son \cite{son}.  Then, several authors devoted
 papers to the evaluations of $F_n(\tau)$ for the integers $n$  not
found in Ramanujan's work ( see \cite{ah,   chua, chua2, ono}).
Later, H. H. Chan, Z.-G. Liu and S. T. Ng \cite{chanliung} gave the
first proof of the entire Theorem \ref{ram}, i.e., (\ref{1eq5}) and
(\ref{1eq6}). By applying the Jacobi imaginary transformation to
(\ref{1eq5}), Chan, Liu and Ng \cite{chanliung} got
\begin{thm} \label{t2}For any positive integer $n$, there exists a
quantity $G_n(\tau)$ such that
\begin{equation*}
\sum_{k=0}^{n-1}\theta^n_3\left(z+\frac{k\pi}{n}\Big|
\tau\right)=G_n(\tau)\theta_3(nz|n\tau). \label{2eq1}
\end{equation*}
\end{thm}

In their paper \cite{bzz}, M. Boon \textit{et al.} proved the
following additive decomposition of $\theta_3(z|\tau)$ (see
\cite[Eq. (7)]{bzz}).
\begin{thm}\label{boon}
For any positive integer $n$, we have
\begin{equation*}
\sum_{k=0}^{n-1}\theta_3(z+k\pi|\tau)= n\theta_3(nz|n^2\tau).
\label{1eqb}
\end{equation*}
\end{thm}
Inspired by \cite{chanliung}  and \cite{bzz}, X.-F. Zeng \cite{zeng}
proved the following important formula unifying Theorem \ref{t2} and
\ref{boon}\ .
\begin{thm}\label{zeng}
For any positive integer $m,n,a$ and $b$ with $a+b=n$, there exists
a quantity $G_{a,b,m,n}(y|\tau)$ such that
\begin{equation*}
\sum_{k=0}^{mn-1}\theta_3^a(z+{y\over a}+{k\pi\over
mn}|\tau)\theta_3^b(z-{y\over a}+{k\pi\over mn}|\tau)=
G_{a,b,m,n}(y|\tau)\theta_3(mnz|m^2n\tau).
\end{equation*}
\end{thm}

Recently, S. H. Chan and Z.-G. Liu \cite{chanliu} eliminated the
unnecessary restrictions in Zeng's Theorem \ref{zeng} and added many
free parameters into it. S. H. Chan and Z.-G. Liu \cite{chanliu}
extended Zeng's theorem to the following more general and concise
form.
\begin{thm} \label{chanliu}
Suppose $y_1,y_2,\cdots,y_n$ are $n$ complex numbers such that $y_1+
y_2+ \cdots+ y_n=0$. Then there exists a quantity
$G_{m,n}(y_1,y_2,\cdots,y_n|\tau)$ such that
\begin{equation*}
\sum_{k=0}^{mn-1}\prod_{j=1}^{n} \theta_3\left(z+y_j+{k\pi\over
mn}\big|\tau\right)=
G_{m,n}(y_1,y_2,\cdots,y_n|\tau)\theta_3(mnz|m^2n\tau). \label{1eql}
\end{equation*}
\end{thm}
Boon \textit{et al.} \cite{bzz} also obtained the following
alternate circular summation formula (see the second idenitity in
\cite[Eq. (8)]{bzz}.
\begin{thm} \label{boona}
For any positive integer $n$, we have
\begin{equation}    \label{boon2}
\sum_{k=0}^{2m-1}(-1)^k\theta_3\left(z+{k\pi\over 2m}|\tau\right)
=2m\theta_2 (2mz|4m^2\tau).
\end{equation}
\end{thm}
Motivated by  \cite{bzz} and \cite{chanliu}, we obtain the following
general alternate circular summation formula of theta functions.

\begin{thm}   \label{fund}
Suppose that $m$ and $n$ are any  positive integers such that $mn$ are even and $y_1, y_2,
\cdots, y_n$ are $n$ complex numbers such that $y_1+ y_2+ \cdots+
y_n=0$. Then we have
\begin{equation}   \label{alter}
 \sum_{k=0}^{mn-1}(-1)^k\prod_{j=1}^n\theta_3\left(z+y_j+{k\pi\over
 mn}|\tau\right)
 =H_{m,n}(y_1,y_2,\cdots,y_n|\tau)\theta_2 (mnz|m^2n\tau),
\end{equation}
where
\begin{equation}     \label{coe}
H_{m,n}(y_1,y_2,\cdots,y_n|\tau) =mnq^{-\frac{m^2n}{8}}
\sum_{\substack{s_1,\cdots,s_n=-\infty\\s_1+\cdots+s_n={mn\over2}}}^\infty
q^{{1\over2}(s_1^2+s_2^2+\cdots+s_n^2)}e^{2i(s_1y_1+s_2y_2+\cdots+s_ny_n)}
.\end{equation}
\end{thm}
Note that if $m$ and $n$ are any  positive integers such that $mn$ are odd, the summations on the left-hand
side of (\ref{alter}) are not ``circular". For convenience, in what
follows, we always denote $H_{m,n}(0,0,\cdots,0|\tau)$  simply as
$H_{m,n}(\tau)$.

 The following sections will be organized as follows. In Section \ref{proof}, the proof of Theorem \ref{fund} will be given.
 In Section \ref{sproof}, we will deduce Theorem \ref{boona} and the following
Corollary \ref{gc} and \ref{2m1} from Theorem \ref{fund}.
\begin{cor}\label{gc} We have
\begin{eqnarray}
&&\sum_{k=0}^{2n-1}(-1)^k\theta_3^{2n}
\left(z+{k\pi\over2n}|\tau\right)=H_{1,2n}(\tau)\theta_2
(2nz|2n\tau). \label{ar14}
\end{eqnarray}
\end{cor}
We think  the following identity is very beautiful.
\begin{cor} \label{2m1}
We have
\begin{equation} \label{2m}
\sum\limits_{k=0}^{2m-1}(-1)^k\theta_3(z+y+{k\pi\over 2m}|\tau)
\theta_3(z-y+{k\pi\over 2m}|\tau) =\begin{cases}
2m\theta_2(2y|2\tau)\theta_2(2mz|2m^2\tau),&\mbox{if $m$ is odd;}\\
2m\theta_3(2y|2\tau)\theta_2(2mz|2m^2\tau),&\mbox{if $m$ is even.}
\end{cases}
\end{equation}
\end{cor}
In Section \ref{sproof}, this identity will be discussed in detail.
We will obtain many modular identities from (\ref{2m}). In
particular, several identities in Ramanujan's notebook will be
recovered.

In Section 4, we will give the proofs of the following two
identities for $(q;q)_\infty^{2n}$, respectively, using Theorem
\ref{fund}.

\begin{cor}\label{q1} We have
\begin{eqnarray}  \label{qq}
(q;q)_\infty^{2n}=q^{-n}~ (q^{2n};q^{2n})_\infty
\sum_{\substack{s_1,s_2,\cdots,s_{2n}=-\infty
\\s_1+s_2+\cdots+s_{2n}=2n}}^\infty
q^{{1\over2}(s_1^2+s_2^2+\cdots+s_{2n}^2)} e^{{\pi
i\over2n}\sum\limits_{l=1}^n{(s_l-s_{n+l})(2l-1)}}.
\end{eqnarray}
\end{cor}

\begin{cor} \label{q2}
We have
\begin{eqnarray}  \label{qqq}
(q^{2n};q^{2n})_\infty^{2n}=q^{-{n\over2}} (q;q)_\infty
\sum_{\substack{s_1,s_2,\cdots,s_{2n}=-\infty
\\s_1+s_2+\cdots+s_{2n}=n}}^\infty
q^{n(s_1^2+s_2^2+\cdots+s_{2n}^2)
+{1\over4}\sum\limits_{l=1}^n{(s_l-s_{n+l})(2l-1)}}.
\end{eqnarray}
\end{cor}

\section{The proof of Theorem \ref{fund}}\label{proof}
\begin{proof}
For any  positive integers  $m$ and $n$  such that $mn$ are even, we set
$$
g(z)=\sum_{k=0}^{mn-1}(-1)^k\prod_{j=1}^n\theta_3 \left({z\over
mn}+y_j+{k\pi\over mn}|{\tau\over m^2n}\right).
$$
Then we find that
\begin{eqnarray*}
g(z+\pi)&=&\sum_{k=0}^{mn-1}(-1)^k\prod_{j=1}^n\theta_3
            \left({z\over mn}+y_j+{(k+1)\pi\over mn}|{\tau\over
            m^2n}\right)\\
        &=& \sum_{k=1}^{mn}(-1)^{k-1}\prod_{j=1}^n\theta_3
            \left({z\over mn}+y_j+{k\pi\over mn}|{\tau\over
            m^2n}\right)\\
        &=&\sum_{k=1}^{mn-1}(-1)^{k-1}\prod_{j=1}^n\theta_3
            \left({z\over mn}+y_j+{k\pi\over mn}|{\tau\over
            m^2n}\right)-
            \prod_{j=1}^n
            \theta_3\left({z\over mn}+y_j+\pi|{\tau\over m^2n}\right)\\
        &=&-\sum_{k=0}^{mn-1}(-1)^{k}\prod_{j=1}^n\theta_3
            \left({z\over mn}+y_j+{k\pi\over mn}|{\tau\over
            m^2n}\right)\\
        &=&-g(z).
\end{eqnarray*}
In the above deduction, we have used the condition that $mn$ is even
and the property of $\theta_3(z|\tau)$ in (\ref{pi}). Using the
property of $\theta_3(z|\tau)$ in (\ref{pitau}), we also have
\begin{eqnarray*}
g(z+\pi\tau)&=&\sum_{k=0}^{mn-1}(-1)^k\prod_{j=1}^n\theta_3
               \left({z\over mn}+y_j+{k\pi\over mn}+{\pi\tau\over mn}
               |{\tau\over m^2n}\right)\\
            &=&\sum_{k=0}^{mn-1}(-1)^k\prod_{j=1}^n
               \left[q^{-{1\over 2n}}
               e^{-2im({{z\over mn}+y_j+{k\pi\over mn}})}
               \theta_3\left({z\over mn}+y_j+{k\pi\over mn}
               |{\tau\over m^2n}\right)\right]\\
            &=&q^{-{1\over 2}}e^{-2iz}\sum_{k=0}^{mn-1}(-1)^k\prod_{j=1}^n
               \theta_3\left({z\over mn}+y_j+{k\pi\over mn}
               |{\tau\over m^2n}\right)\\
           &=&q^{-{1\over 2}}e^{-2iz}g(z).
\end{eqnarray*}
Therefore the function
\begin{eqnarray*}
F(z)=\frac{g(z)}{\theta_2(z|\tau)}
\end{eqnarray*}
is an elliptic function of $z$. It is well-known that
$\theta_2(z|\tau)$ has only a simple zero at $z={\pi\over2}$ in the
period parallelogram. This shows that $F(z)$ has at most one pole in
its period parallelogram. Hence $F(z)$ is independent of $z$, say,
it equals $S(y_1,y_2,\cdots,y_n|\tau). $ It follows that
\begin{equation}  \label{con}
\sum_{k=0}^{mn-1}(-1)^k\prod_{j=1}^n \theta_3\left({z\over
mn}+y_j+{k\pi\over mn} |{\tau\over m^2n}\right)
=S(y_1,y_2,\cdots,y_n|\tau)\theta_2 (z|\tau).
\end{equation}
Replacing $z$ and $\tau$ in the above identity by $mnz$ and
$m^2n\tau$, respectively,  gives (\ref{alter}).

Now it remains to determine $S(y_1,y_2,\cdots,y_n|\tau)$. To
complete this, we compare the coefficients of $e^{iz}$ on both sides
of (\ref{con}). Using the definition of $\theta_3(z|\tau)$, we
easily get
\begin{eqnarray*}
\lefteqn{ \prod_{j=1}^n
 \theta_3\left({z\over mn}+y_j+{k\pi\over mn}|{\tau\over
 m^2n}\right)}\hspace{1cm}\\
&=&\sum_{s_1,\cdots,s_n=-\infty}^{\infty}
q^{\frac{1}{2m^2n}(s_1^2+\cdots+s_n^2)}
  e^{\frac{2i(z+k\pi)}{mn}(s_1+\cdots+s_n)+2i(s_1y_1+\cdots+s_ny_n)}.\\
\end{eqnarray*}
Thence the coefficient of $e^{iz}$ of the above series is
$$
(-1)^k
\sum_{\substack{s_1,\cdots,s_n=-\infty\\s_1+\cdots+s_n={mn\over2}}}^\infty
q^{{1\over
2m^2n}(s_1^2+s_2^2+\cdots+s_n^2)}e^{2i(s_1y_1+s_2y_2+\cdots+s_ny_n)}.
$$
The coefficient of $e^{iz}$ on the left-hand side of (\ref{con}) is
$$
mn
\sum_{\substack{s_1,\cdots,s_n=-\infty\\s_1+\cdots+s_n={mn\over2}}}^\infty
q^{{1\over
2m^2n}(s_1^2+s_2^2+\cdots+s_n^2)}e^{2i(s_1y_1+s_2y_2+\cdots+s_ny_n)}.
$$
From the definition of $\theta_2(z|\tau)$, we have
$$
q^{1\over8}S(y_1,y_2,\cdots,y_n|\tau)=mn
\sum_{\substack{s_1,\cdots,s_n=-\infty\\s_1+\cdots+s_n={mn\over2}}}^\infty
q^{{1\over
2m^2n}(s_1^2+s_2^2+\cdots+s_n^2)}e^{2i(s_1y_1+s_2y_2+\cdots+s_ny_n)}.
$$
This is (\ref{coe}), where
$H(y_1,y_2,\cdots,y_n|\tau)=S(y_1,y_2,\cdots,y_n|m^2n\tau)$. This
completes the proof.
\end{proof}

In the above  proof, we have used the fact that an elliptic function
with at most one pole in its period parallelogram is a constant
(see, for example, \cite[p. 432]{whwa}).

Now, we split Theorem \ref{fund} into two cases according to the
parity of $m$ and $n$. Replacing $m$ in Theorem \ref{fund} by $2m$
gives

\textbf{\textit{Case 1 of Theorem \ref{fund}.}}~ Suppose that $m$
and $n$ are any positive integers and $y_1, y_2, \cdots, y_n$ are
$n$ complex numbers such that $y_1+ y_2+ \cdots+ y_n=0$. Then we
have
\begin{equation}   \label{alter1}
\sum_{k=0}^{2mn-1}(-1)^k\prod_{j=1}^n\theta_3\left(z+y_j+{k\pi\over
2mn}|\tau\right)=H_{2m,n}(y_1,y_2,\cdots,y_n|\tau)\theta_2
(2mnz|4m^2n\tau).
\end{equation}

We replace $n$ in Theorem \ref{fund} by $2n$ to obtain

\textbf{\textit{Case 2 of Theorem \ref{fund}.}}~
 Suppose that $m$ and $n$ are
any positive integers and $y_1, y_2, \cdots, y_{2n}$ are $2n$
complex numbers such that $y_1+ y_2+ \cdots+ y_{2n}=0$. Then we have
\begin{equation}   \label{alter2}
\sum_{k=0}^{2mn-1}(-1)^k\prod_{j=1}^{2n}\theta_3\left(z+y_j+{k\pi\over
2mn}|\tau\right)=H_{m,2n}(y_1,y_2,\cdots,y_{2n}|\tau)\theta_2
(2mnz|2m^2n\tau).
\end{equation}

Note that both $H_{2m,n}(y_1,y_2,\cdots,y_n|\tau)$ in (\ref{alter1})
and $H_{m,2n}(y_1,y_2,\cdots,y_{2n}|\tau)$ in (\ref{alter2}) are
defined by (\ref{coe}).

\section{Some special cases of the alternate summation\\
 formula 
 }\label{sproof}

\begin{proof}[\textbf{Proof of Theorem \ref{boona}}]
 Putting $n=1$ in (\ref{alter1}) gives Theorem \ref{boona}\ .
\end{proof}
\begin{proof}[\textbf{Proof of Corollary \ref{gc}}]
 Setting $m=1$ and $y_j=0$ in (\ref{alter2}) gives (\ref{ar14}).
\end{proof}
Replacing $z$ in (\ref{ar14}) by $z+{\pi+\pi\tau\over2}$ and then
using (\ref{pi/2}) and (\ref{pitau/2}), we get
$$
\sum_{k=0}^{2n-1}\theta_1^{2n}\left(z+{k\pi\over \label{ar2}
2n}|\tau\right)=H_{1,2n}(\tau)\theta_3 (2nz|2n\tau).
$$
The left-hand side of this identity is the same as that of the even
case of \cite[Thm. 4.2]{chanliung}. But there is a serious misprint
on the left-hand side of the identity in \cite[Thm. 4.2]{chanliung}.

\begin{proof}[\textbf{Proof of Corollary \ref{2m1}}]
From (\ref{coe}),  we have
 \begin{eqnarray*}
H_{m,2}(y,-y|\tau)&=&2mq^{-{m^2\over4}}\sum_{s=-\infty}^{\infty}
   q^{s^2+(m-s)^2\over2}e^{2iy[s-(m-s)]} \notag\\
&=&2mq^{m^2\over4}e^{-2imy}\sum_{s=-\infty}^{\infty}
   q^{s^2}e^{2is(y-m\pi\tau)} \notag\\
&=&2mq^{m^2\over4}e^{-2imy}\theta_3(2y-m\pi\tau|2\tau) \notag\\
&=&\begin{cases}
2m\theta_2(2y|2\tau),~&\mbox{if $m$ is odd;}\\
2m\theta_3(2y|2\tau),~&\mbox{if $m$ is even.}
\end{cases}
\notag
\end{eqnarray*}
Then setting $n=1$ in (\ref{alter2}), we obtain (\ref{2m}).
\end{proof}

Taking $m=1$ in (\ref{2m}) and then using (\ref{pi/2}) gives
\begin{prop}We have
\begin{eqnarray*}
&&\theta_3(z+y|\tau) \theta_3(z-y|\tau)-\theta_4(z+y|\tau)
\theta_4(z-y|\tau) =2\theta_2(2y|2\tau)\theta_2(2z|2\tau).
\end{eqnarray*}
\end{prop}
This identity is equivalent to \cite[Eq. (16)]{enn} and \cite[Eq.
(1.3d), p327]{shen}. J. A. Ewell \cite[Eq. (1. 10)]{ewell} deduced a
sextuple product identity from this one. Z.-G. Liu and X.-M. Yang
\cite[Eq. (1.11) in Thm. 4]{liuyang} also deduced this identity from
the Schr\"oter formula. Using (\ref{pi/2}), (\ref{pitau/2}) and
Jacobi imaginary transformation to this identity, all the rest
identities in \cite[Thm. 4]{liuyang} can be deduced.

Setting $m=2$ in (\ref{2m}), and then, using (\ref{pi/2}), we get
the following remarkable identity, which contains many interesting
special cases.
\begin{prop}\label{4z} We have
\begin{eqnarray}
\lefteqn{\begin{array}{l} \theta_3(z+y|\tau)\theta_3(z-y|\tau)
-\theta_3(z+y+{\pi\over4}|\tau)\theta_3(z-y+{\pi\over4}|\tau)\quad
\notag
\\+\theta_4(z+y|\tau)\theta_4(z-y|\tau)
-\theta_4(z+y+{\pi\over4}|\tau)\theta_4(z-y+{\pi\over4}|\tau)\notag
\end{array}}\hspace{5.5cm}\\
&=&4\theta_3(2y|2\tau)\theta_2(4z|8\tau). \label{ord4}
\end{eqnarray}
\end{prop}
We will show that the following identities can be deduced from
(\ref{ord4}). The first four identities can also be found in Berndt
\cite[Entry 25, p. 40]{berndt}. 
\begin{cor}
We have
\renewcommand{\labelenumi}{\rm (\alph{enumi})}
\begin{enumerate}
\item
$\varphi(q)\psi(q^2)=\psi^2(q)$,
\item
$\varphi(q)-\varphi(-q)=4q\psi(q^8)$,
\item
$\varphi(q)+\varphi(-q)=2\varphi(q^4)$,
\item
$\varphi^2(q)-\varphi^2(-q)=8q\psi^2(q^2)$,
\item
$\psi^2(q)-\varphi(-q)\psi(q^2)=4q\psi(q^2)\psi(q^8)$,
\item
$\psi^2(q)+\varphi(-q)\psi(q^2)=2\psi(q^2)\varphi(q^4)$.
\end{enumerate}    \label{mod}
\end{cor}
\begin{proof}[\textbf{Proof}]
In the following proofs, (\ref{thetap})--(\ref{ssf}) will be used
often. We only prove (a) in detail.
\renewcommand{\labelenumi}{\rm (\alph{enumi})}
\begin{enumerate}
\item
Setting $y=0$ and $z={\pi\tau\over2}$ in (\ref{ord4}), we have
\begin{equation*}
 \theta_3^2({\pi\tau\over2}|\tau)
 -\theta_3^2({\pi\tau\over2}
 +{\pi\over4}|\tau)
 +\theta_4^2({\pi\tau\over2}|\tau)
 -\theta_4^2({\pi\tau\over2}+{\pi\over4}|\tau)
 =4\theta_3(0|2\tau)\theta_2(2\pi\tau|8\tau).
\end{equation*}
Applying (\ref{pitau/2}) to the above identity gives
\begin{equation*}
 \theta_2^2(0|\tau)
 -i\theta_2^2({\pi\over4}|\tau)
 +\theta_1^2(0|\tau)
 +i\theta_1^2({\pi\over4}|\tau)
 =4q^{1\over4}\theta_3(0|2\tau)\theta_2(2\pi\tau|8\tau).
\end{equation*}
Note that $\theta_1(0|\tau)=0$ by (\ref{thetap}) and
$\theta_2^2({\pi\over4}|\tau)=\theta_1^2({\pi\over4}|\tau)$ by
(\ref{pi/2}). Then the above identity reduces to
\begin{equation*}
\theta_2^2(0|\tau)
=4q^{1\over4}\theta_3(0|2\tau)\theta_2(2\pi\tau|8\tau).
\end{equation*}
Combining (\ref{thetap}) and (\ref{sf}) with the above identity
gives what we need.
\item
Take $y={\pi\over4}$ and $z=0$ in (\ref{ord4}).
\item
Take $y={\pi\over4}$ and $z=\pi\tau$ in (\ref{ord4}).
\item
This identity can be proved by either combining the above three
identities or taking $y=z={\pi\tau\over2}$ in (\ref{ord4}).
\item
Take $y={\pi\tau\over2}$ and $z=0$ in (\ref{ord4}).
\item
Take $y={\pi\tau\over2}$ and $z=\pi\tau$ in (\ref{ord4}).
\end{enumerate}
\end{proof}
Obviously, from (\ref{ord4}), more modular identities can be
deduced.
\section{The proofs of the two formulaes for  $(q;q)^{2n}_\infty$}
In this section, we prove the two formulaes for $(q;q)^{2n}_\infty$
from two special cases of $H_{m,n}(y_1,y_2,\cdots,y_n)$ in Theorem
\ref{fund},  respectively.

\begin{proof}[\textbf{Proof of Corollary \ref{q1}}]
 Note that
\begin{eqnarray*}   
\prod_{j=1}^{n}\theta_3\left(z+{(2j-1)\pi\over4n}|\tau\right)
\theta_3\left(z-{(2j-1)\pi\over4n}|\tau\right)
&=&\frac{(q;q)_{\infty}^{2n}}{(q^{2n};q^{2n})_{\infty}}
\theta_3(2nz|2n\tau).
\end{eqnarray*}
In (\ref{alter2}), we set $y_j={(2j-1)\pi\over4n}$ and
$y_{n+j}=-{(2j-1)\pi\over4n}$ for $1\leq j\leq n$. Then the left-hand side of (\ref{alter2}) equals
\begin{eqnarray*}
\sum_{k=0}^{2mn-1}\prod_{j=1}^{n}
\theta_3\left(z+{(2j-1)\pi\over4n}+{k\pi\over2mn}|\tau\right)
\theta_3\left(z-{(2j-1)\pi\over4n}+{k\pi\over2mn}|\tau\right)\\
=\frac{(q;q)_{\infty}^{2n}}{(q^{2n};q^{2n})_{\infty}}
\sum_{k=0}^{2mn-1}(-1)^k\theta_3\left(2nz+{k\pi\over
m}|2n\tau\right).
\end{eqnarray*}
Hence we have
\begin{eqnarray}
\lefteqn{\sum_{k=0}^{2mn-1}(-1)^k\theta_3\left(2nz+{k\pi\over
m}|2n\tau\right)}\hspace{0cm}    \notag\\
&=&\begin{array}{l}
\frac{(q^{2n};~q^{2n})_{\infty}}{(q;~q)_{\infty}^{2n}}
H_{m,2n}\left({\pi\over4n},{3\pi\over4n},\cdots,{(2n-1)\pi\over4n},
-{\pi\over4n},-{3\pi\over4n},\cdots,-{(2n-1)\pi\over4n}
|\tau\right)\theta_2 (2mnz|2m^2n\tau). \notag\\
\end{array}\\   \label{ilte22}
\end{eqnarray}
The left-hand side of the above identity equals
\begin{eqnarray*}
\sum\limits_{l=0}^{2n-1}\sum\limits_{k=lm}^{(l+1)m-1}(-1)^k
\theta_3\left(2nz+{k\pi\over m}|2n\tau\right) &=&
\sum\limits_{l=0}^{2n-1}\sum\limits_{k=0}^{m-1}(-1)^{k+lm}
\theta_3\left(2nz+l\pi+{k\pi\over m}|2n\tau\right)\\
&=&\sum\limits_{l=0}^{2n-1}\sum\limits_{k=0}^{m-1}(-1)^{k+lm}
\theta_3\left(2nz+{k\pi\over m}|2n\tau\right).
\end{eqnarray*}
Substitute the above identity back into (\ref{ilte22}) to obtain
\begin{eqnarray*}   \label{inalter2}
\lefteqn{\sum\limits_{l=0}^{2n-1}\sum\limits_{k=0}^{m-1}(-1)^{k+lm}
\theta_3\left(2nz+{k\pi\over m}|2n\tau\right)}\hspace{0cm}\\
&=&
\begin{array}{l}
\frac{(q^{2n};~q^{2n})_{\infty}}{(q;~q)_{\infty}^{2n}}
H_{m,2n}\left({\pi\over4n},{3\pi\over4n},\cdots,{(2n-1)\pi\over4n},
-{\pi\over4n},-{3\pi\over4n},\cdots,-{(2n-1)\pi\over4n}
|\tau\right)\theta_2 (2mnz|2m^2n\tau).
\end{array}
\end{eqnarray*}
Note that the left-hand side of the above identity equals $0$ when
$m$ is odd. When $m$ is even, replace $m$ by $2m$ in the above
identity to get
\begin{eqnarray*}   \label{inalter2}
\lefteqn{2n\sum\limits_{k=0}^{2m-1}(-1)^{k}
\theta_3\left(2nz+{k\pi\over 2m}|2n\tau\right)}\hspace{15cm}\\
=
\begin{array}{l}
\frac{(q^{2n};~q^{2n})_{\infty}}{(q;~q)_{\infty}^{2n}}
H_{2m,2n}\left({\pi\over4n},{3\pi\over4n},\cdots,{(2n-1)\pi\over4n},
-{\pi\over4n},-{3\pi\over4n},\cdots,-{(2n-1)\pi\over4n}
|\tau\right)\theta_2(4mnz|8m^2n\tau).
\end{array}
\end{eqnarray*}
By (\ref{boon2}), the left-hand side of the above identity equals
$4mn\theta_2(4mnz|8m^2n\tau)$. Substituting this back into the above
identity, and then cancelling $\theta_2(4mnz|8m^2n\tau)$ from both
sides of the resulting equation, we obtain
\begin{eqnarray*}   \label{hm2n}
H_{2m,2n}\left({\pi\over4n},{3\pi\over4n},\cdots,{(2n-1)\pi\over4n},
-{\pi\over4n},-{3\pi\over4n},\cdots,-{(2n-1)\pi\over4n}
|{\tau}\right)=
\frac{4mn(q;~q)_{\infty}^{2n}}{(q^{2n};~q^{2n})_{\infty}}.
\end{eqnarray*}
Then, combining with (\ref{coe}), we obtain
\begin{eqnarray*}
(q;q)_\infty^{2n}=q^{-m^2n}~ (q^{2n};q^{2n})_\infty
\sum_{\substack{s_1,s_2\cdots,s_{2n}=-\infty
\\s_1+s_2\cdots+s_{2n}=2mn}}^\infty
q^{{1\over2}(s_1^2+s_2^2+\cdots+s_{2n}^2)} e^{{\pi
i\over2n}\sum\limits_{l=1}^n{(s_l-s_{n+l})(2l-1)}}.
\end{eqnarray*}

Note the left-hand side of the above identity is independent of $m$.
Setting $m=1$, we get (\ref{qq}), which ends the proof.
\end{proof}

\begin{proof}[\textbf{Proof of Corollary \ref{q2}}]
In (\ref{alter2}), we put $y_j={(2j-1)\pi\tau\over4n}$ and
$y_{n+j}=-{(2j-1)\pi\tau\over4n}$ for $1\leq j\leq n$. Note the fact
that
\begin{equation*}  
\prod_{j=1}^{n}\theta_3\left(x+{(2j-1)\pi\tau\over4n}|\tau\right)
\theta_3\left(x-{(2j-1)\pi\tau\over4n}|\tau\right)
={(q;q)_{\infty}^{2n}\over(q^{1\over2n};q
^{1\over2n})_{\infty}}\theta_3\left(x|{\tau\over2n}\right).
\end{equation*}
%
By similar computation, we obtain
\begin{eqnarray*}
(q^{2n};q^{2n})_\infty^{2n}=q^{-{m^2n\over2}} (q;q)_\infty
\sum_{\substack{s_1,s_2\cdots,s_{2n}=-\infty
\\s_1+s_2\cdots+s_{2n}=mn}}^\infty
q^{n(s_1^2+s_2^2+\cdots+s_{2n}^2)
+{1\over4}\sum\limits_{l=1}^n{(s_l-s_{n+l})(2l-1)}}.
\end{eqnarray*}
Setting $m=1$, we get (\ref{qqq}). This achieves the proof.
\end{proof}


%
%
%
%
%


\begin{thebibliography}{}
\bibitem{ah}
S. Ahlgren, The sixth, eighth, ninth and tenth powers of Ramanujan
theta function, Proc. Amer. Math. Soc. 128 (2000) 1333--1338.

\bibitem{berndt}
B. C. Berndt, Ramanujan's Notebook, Part III, Springer--Verlag, 1991.

\bibitem{bzz}
M. Boon, M. L. Glasser, J. Zak, J. Zucker, Additive decompositions
of $\theta$-functions of multiple arguments, J. Phys. A 15 (1982)
3439--3440.

\bibitem{chanliung}
H. H. Chan, Z.-G. Liu, S. T. Ng, Circular summation of theta
functions in Ramanujan's Lost Notebook, J. Math. Anal. Appl. 316
(2006) 628--641.

\bibitem{chanliu}
S. H. Chan, Z.-G. Liu, On a new circular summation of theta
functions, J. Number Theory 130 (2010) 1190--1196.

\bibitem{chua}
K. S. Chua, Circular summation of the 13th powers of
Ramanujan's theta function, Ramanujan J. 5 (2001) 353--354.

\bibitem{chua2}
K. S. Chua, The root lattice $A_n^*$ and Ramanujan's circular
summation of theta functions, Proc. Amer. Math. Soc. 130 (2002)
1--8.

\bibitem{enn}
A. Enneper, Elliptische Function: Theorie and Geschichte, Louis
Nebert, Halle, 1890.

\bibitem{ewell}
J. A. Ewell, Arithmetical consequences of a sextuple product
identity, Rocky Mountain J. of Math. 25 (1997) 1287--1293.


\bibitem{liuyang}
Z.-G. Liu, X.-M. Yang, On the Schr\"oter formula for theta
functions, Int. J. Number Theory 5 (8) (2009)  1477--1488.


\bibitem{ono}
K. Ono, On the circular summation of the eleventh powers of
Ramanujan's theta function, J. Number Theory 76 (1999) 62--65.

\bibitem{ram}
S. Ramanujan, The Lost Notebook and Other Unpublished Papers,
Narosa, New Delhi, 1988.

\bibitem{ran}
S. S. Rangachari, On a result of Ramanujan on theta functions, J.
Number Theory 48 (1994) 364--372.

\bibitem{shen}
L. -C. Shen, On the additive formula of the theta functions and a
collection of Lambert series pertaining to the modular equations of
degree 5, Trans. Amer. Math. Soc. 345 (1994) 323--345.


\bibitem{son}
S. H. Son, Circular summation of theta functions in Ramanujan's Lost
Notebook, Ramanujan J. 8 (2004) 235--272.

\bibitem{whwa}
 E. T. Whittaker, G. N. Watson, A Course of Modern Analysis, 4th ed.,
 Cambridge University Press, Cambridge, 1996.


\bibitem{zeng}
X.-F. Zeng, A generalized circular summation of theta function and
its application, J. Math. Anal. Appl. 356 (2009) 698--703.
\end{thebibliography}
\end{document}